\documentclass[10pt]{amsart}

\usepackage{amsmath,amsthm,amsfonts,amscd,amssymb,soul,xcolor}
\usepackage{hyperref,wasysym}
\usepackage{verbatim}
\usepackage{amssymb}
\usepackage{graphicx}
\usepackage{color}

\newtheorem{thm}{Theorem}[section]
\newtheorem{cor}[thm]{Corollary}
\newtheorem{lem}[thm]{Lemma}
\newtheorem{prop}[thm]{Proposition}

\theoremstyle{definition}

\theoremstyle{remark}
\newtheorem{rem}{Remark}[section]

\usepackage{marginnote}

\begin{document}

\title{On lattices generated by algebraic conjugates of prime degree}

\author{Lenny Fukshansky}
\author{Evelyne Knight}

\address{Department of Mathematics, 850 Columbia Avenue, Claremont McKenna College, Claremont, CA 91711}
\email{lenny@cmc.edu}
\address{Department of Mathematics, Pomona College, 610 N. College Ave, Claremont, CA 91711, USA}
\email{eski2022@mymail.pomona.edu}

\subjclass[2020]{Primary: 11H06, 11R04, 11R06, 11R32}
\keywords{well-rounded lattices, Pisot polynomials, algebraic conjugates}

\begin{abstract} We consider Euclidean lattices spanned by images of algebraic conjugates of an algebraic number under Minkowski embedding, investigating their rank, properties of their automorphism groups and sets of minimal vectors. We are especially interested in situations when the resulting lattice is well-rounded. We show that this happens for large Pisot numbers of prime degree, demonstrating infinite families of such lattices. We also fully classify well-rounded lattices from algebraic conjugates in the 2-dimensional case and present various examples in the 3-dimensional case. Finally, we derive a determinant formula for the resulting lattice in the case when the minimal polynomial of an algebraic number has its Galois group of a particular type.
\end{abstract}

\maketitle

\def\A{{\mathcal A}}
\def\B{{\mathcal B}}
\def\C{{\mathcal C}}
\def\D{{\mathcal D}}
\def\F{{\mathcal F}}
\def\x{{\mathcal H}}
\def\I{{\mathcal I}}
\def\J{{\mathcal J}}
\def\K{{\mathcal K}}
\def\L{{\mathcal L}}
\def\M{{\mathcal M}}
\def\N{{\mathcal N}}
\def\O{{\mathcal O}}
\def\R{{\mathcal R}}
\def\s{{\mathcal S}}
\def\V{{\mathcal V}}
\def\W{{\mathcal W}}
\def\X{{\mathcal X}}
\def\Y{{\mathcal Y}}
\def\H{{\mathcal H}}
\def\Z{{\mathcal Z}}
\def\OO{{\mathcal O}}
\def\BB{{\mathbb B}}
\def\cee{{\mathbb C}}
\def\EE{{\mathbb E}}
\def\Nn{{\mathbb N}}
\def\pee{{\mathbb P}}
\def\que{{\mathbb Q}}
\def\real{{\mathbb R}}
\def\zed{{\mathbb Z}}
\def\hyp{{\mathbb H}}
\def\aa{{\mathfrak a}}
\def\HH{{\mathfrak H}}
\def\qbar{{\overline{\mathbb Q}}}
\def\eps{{\varepsilon}}
\def\ahat{{\hat \alpha}}
\def\bhat{{\hat \beta}}
\def\gt{{\tilde \gamma}}
\def\h{{\tfrac12}}
\def\be{{\boldsymbol e}}
\def\bei{{\boldsymbol e_i}}
\def\bff{{\boldsymbol f}}
\def\ba{{\boldsymbol a}}
\def\bb{{\boldsymbol b}}
\def\bc{{\boldsymbol c}}
\def\bm{{\boldsymbol m}}
\def\bn{{\boldsymbol n}}
\def\bk{{\boldsymbol k}}
\def\bi{{\boldsymbol i}}
\def\bl{{\boldsymbol l}}
\def\bq{{\boldsymbol q}}
\def\bu{{\boldsymbol u}}
\def\bt{{\boldsymbol t}}
\def\bs{{\boldsymbol s}}
\def\bv{{\boldsymbol v}}
\def\bw{{\boldsymbol w}}
\def\bx{{\boldsymbol x}}
\def\bX{{\boldsymbol X}}
\def\bz{{\boldsymbol z}}
\def\bwy{{\boldsymbol y}}
\def\bY{{\boldsymbol Y}}
\def\bL{{\boldsymbol L}}
\def\baa{{\boldsymbol\alpha}}
\def\bbb{{\boldsymbol\beta}}
\def\bgg{{\boldsymbol\gamma}}
\def\bet{{\boldsymbol\eta}}
\def\bxi{{\boldsymbol\xi}}
\def\bo{{\boldsymbol 0}}
\def\bol{{\boldkey 1}_L}
\def\ep{\varepsilon}
\def\p{\boldsymbol\varphi}
\def\q{\boldsymbol\psi}
\def\rank{\operatorname{rank}}
\def\aut{\operatorname{Aut}}
\def\lcm{\operatorname{lcm}}
\def\sgn{\operatorname{sgn}}
\def\spn{\operatorname{span}}
\def\md{\operatorname{mod}}
\def\Norm{\operatorname{Norm}}
\def\dim{\operatorname{dim}}
\def\det{\operatorname{det}}
\def\Vol{\operatorname{Vol}}
\def\rk{\operatorname{rk}}
\def\Gal{\operatorname{Gal}}
\def\WR{\operatorname{WR}}
\def\WO{\operatorname{WO}}
\def\GL{\operatorname{GL}}
\def\pr{\operatorname{pr}}
\def\Tr{\operatorname{Tr}}
\def\dd{\partial}
\def\itt{\operatorname{int}}
\def\Ar{\operatorname{Area}}
\def\Aut{\operatorname{Aut}}

\section{Introduction}
\label{intro}

Algebraic constructions of Euclidean lattices have received a great deal of attention in lattice theory over the past years. Perhaps the most commonly studied construction is that of {\it ideal lattices} from algebraic number fields, thoroughly described by E. Bayer-Fluckiger (see, e.g.,~\cite{bayer}): a Euclidean lattice is obtained from an ideal in the ring of integers of a number field in question via an application of Minkowski embedding. The actively studied properties of lattices arising from such constructions include number and configurations of minimal vectors, as well as the  size and structure of the automorphism group. The underlying algebraic structure allows for simpler description and informs the geometric properties of the resulting lattices. A great deal of attention in the recent years was specifically devoted to the {\it well-rounded} ideal lattices, i.e. ideal lattices containing full linearly independent sets of minimal vectors (see, e.g., \cite{lf_kp}, \cite{lf_wr-2}, \cite{batson}, \cite{damir_karpuk}, \cite{anitha}, \cite{ha-1}, \cite{ha-2}, \cite{alves}). Applications of well-rounded ideal lattices to coding theory and cryptography were considered, for instance, in~\cite{micc}, \cite{damir-2} and~\cite{camilla}.

A different construction of algebraic well-rounded lattices from cyclic number fields of odd prime degree has been proposed in~\cite{robson-1} and~\cite{robson-2}: instead of applying Minkowski embedding to an ideal (as in the construction of ideal lattices), the authors considered $\zed$-modules in the ring of integers spanned by the images of an algebraic integer under the action of the cyclic Galois group. In the present paper, we want to generalize this construction, viewing it from a somewhat different angle. 

Let
$$f(x) = x^n + a_{n-1}x^{n-1} + \dots + a_1x + a_0 \in \zed[x]$$
be an irreducible polynomial of degree $n \geq 2$ and let 
$$\alpha_1, \alpha_2,\dots, \alpha_n$$
be roots of $f(x)$. Notice that we take $f(x)$ to be a monic polynomial; this is only to simplify the statements of our results. Let $K$ be its splitting field and $G$ its Galois group over $\que$. Then 
$$d := [K:\que] = |G| \leq n!.$$
Write $G = \left\{ \sigma_1,\dots,\sigma_d \right\}$, then for each $1 \leq i \leq n$ and $1 \leq j \leq d$, $\sigma_j(\alpha_i) = \alpha_{i_j}$ for some corresponding $1 \leq i_j \leq n$. Further, 
\begin{equation}
\label{sum_a}
-a_{n-1} = \alpha_1 + \dots + \alpha_n \in \zed.
\end{equation}
Define
$$\M_f = \spn_{\zed} \left\{ \alpha_1,\dots,\alpha_n \right\} \subset K,$$
so $\M_f$ is a $\zed$-module of rank $\leq n$ generated by all the conjugates of $\alpha$. Notice that $\M_f$ is closed under the action of $G$, since automorphisms of the Galois extension $K/\que$ simply permute the roots of an irreducible polynomial. 

We can identify elements of $G$ with the embeddings of $K$ into $\cee$, where $r_1$ of them are real and $2r_2$ are complex, coming in conjugate pairs, i.e., $d=r_1+2r_2$. The Minkowski embedding
$$\Sigma_K = (\sigma_1,\dots,\sigma_d) : K \hookrightarrow \cee^d$$
takes $K$ into the space $K_{\real} := K \otimes_{\que} \real$, which can be viewed as a subspace of
$$\left\{ (\bx,\bwy) \in \real^{r_1} \times \cee^{2r_2} : y_{r_2+j} = \bar{y}_j\ \forall\ 1 \leq j \leq r_2 \right\} \cong \real^{r_1} \times \cee^{r_2} \subset \cee^d,$$
where in the last containment each copy of $\real$ is identified with the real part of the corresponding copy of $\cee$. For each $\alpha \in K$, we write $\baa$ for its image $\Sigma_K(\alpha)$ in $K_{\real}$. Then $K_{\real}$ is a $d$-dimensional Euclidean space with the symmetric bilinear form induced by the trace form on $K$:
$$\left< \baa,\bbb \right> := \Tr_K(\alpha \bar{\beta}) \in \real,$$
for every $\alpha,\beta \in K$, where $\Tr_K$ is the number field trace on $K$. The Euclidean norm $\|\baa\|$ on $K_{\real}$ is given by $\sqrt{\left< \baa,\baa \right>}$ and the angle $\theta$ between $\baa$ and $\bbb$ is given by
$$\theta = \cos^{-1} \left( \frac{\left< \baa,\bbb \right>}{\|\baa\| \|\bbb\|} \right).$$

Define $L_f := \Sigma_K(\M_f)$, which is a Euclidean lattice in $K_{\real}$. Then
$$L_f = \spn_{\zed} \left\{ \baa_1,\dots,\baa_n \right\}.$$
More generally, we will use boldface letters to denote vectors in $L_f$ that are images of the corresponding algebraic numbers in $\M_f$ under the Minkowski embedding $\Sigma_K$. We write $\Aut(L_f)$ for the automorphism group of the lattice $L_f$, i.e., the group of isometries of the trace-induced bilinear form $\left<\ ,\ \right>$ preserving $L_f$. Our first simple observation concerns this automorphism group.

\begin{prop} \label{aut} The group $G$ is a subgroup of $\Aut(L_f)$.
\end{prop}

\noindent
We prove this proposition in Section~\ref{auto_min}. Notice that it guarantees that our construction $L_f$ produces lattices with relatively large automorphism groups; indeed, a generic lattice has only two automorphisms: $\bx \mapsto \pm \bx$. In fact, unlike our construction $L_f$, ideal lattices are not necessarily closed under the action of the Galois group. We also define the {\it minimal norm} of $L_f$ to be
$$|L_f| = \min \left\{ \|\bx\| : \bx \in L_f \setminus \{ \bo \} \right\}$$
and the {\it set of minimal vectors} of $L_f$ to be
$$S(L_f) = \left\{ \bx \in L_f : \|\bx\| =|L_f| \right\}.$$
With this notation, we can now state the following theorem.

\begin{thm} \label{auto} Suppose $n$ is prime. Then:
\begin{enumerate}

\item The lattice $L_f$ has rank $n$ if $a_{n-1} \neq 0$ and rank $n-1$ if $a_{n-1}=0$.

\item If $|L_f| < \sqrt{d} |a_{n-1}|$, then the cardinality of $S(L_f)$ is divisible by $n$.

\end{enumerate}
\end{thm}

\noindent
A lattice $L_f$ is called {\it well-rounded} (WR) if $S(L_f)$ contains at least $n$ linearly independent vectors, i.e., if
$$\spn_{\real} L_f = \spn_{\real} S(L_f).$$

\begin{cor} \label{n_WR} Let $n = d$ be prime, $|L_f| < \sqrt{n} |a_{n-1}|$ and let
$$\bbb = \sum_{k=1}^n c_k \baa_k \in S(L_f).$$
If $\sum_{k=1}^n c_k \neq 0$ then $L_f$ is WR.
\end{cor}

\noindent
We can give a more explicit criterion for well-roundedness in the three-dimensional case.

\begin{cor} \label{cor:n_WR-cubic}
Let $n=d=3$ and assume that $|L_f|<\sqrt 3|a_{2}|$. Then:
\begin{enumerate}

\item If $\sum_{i < j}\alpha_i\alpha_j<\frac{1}{2}\sum_{i\leq 3} \alpha_i^2$, then the lattice $L_f$  is WR. 

\item If $\left| \sum_{i < j}\alpha_i\alpha_j \right| < \frac{1}{3}\sum_{i\leq 3} \alpha_i^2$, then $\baa_1,\baa_2,\baa_3 \in S(L_f)$.
\end{enumerate}
\end{cor}

\noindent
We prove Theorem~\ref{auto} and Corollaries~\ref{n_WR} and ~\ref{cor:n_WR-cubic} in Section~\ref{auto_min}. Our main tool is a result of Dubickas~\cite{dub} on the degree of integer linear forms in algebraic conjugates. We separately consider the two-dimensional situation in Section~\ref{2dim}, where the results can be made very explicit. In particular, we fully classify planar lattices of the form $L_f$ in Theorem~\ref{Lf2}.
\smallskip

To state our main result, we need some additional notation. Let $\B = \{ \bb_1,\dots,\bb_n \}$ be an ordered basis for a lattice $L$ in~$\spn_{\real} L$, and define a sequence of angles $\theta_1,\dots,\theta_{n-1}$ as follows: each $\theta_i$ is the angle between $\bb_{i+1}$ and the subspace 
$$\spn_{\real} \{ \bb_1,\dots,\bb_i \}.$$
Then each $\theta_i \in [0,\pi/2]$ and we say that $\B$ is a {\it weakly nearly orthogonal} basis if $\theta_i \geq \pi/3$ for each $1 \leq i \leq n-1$. A basis $\B$ is called {\it nearly orthogonal} if every ordering of it is weakly nearly orthogonal. If $L$ has such a basis, we say that $L$ is a nearly orthogonal lattice. Nearly orthogonal lattices were introduced in~\cite{baraniuk}, where they were applied to the problem of image compression. Nearly orthogonal lattices that are also well-rounded have been studied in~\cite{lf_dk-1}.

The class of generic well-rounded (GWR) lattices was introduced in~\cite{camilla}: a WR lattice $L$ of rank $n$ is called {\it generic WR} if $|S(L)| = 2n$, i.e., if $L$ contains the minimal possible number of minimal vectors for a WR lattice. GWR lattices are important for their application in coding theory and cryptography since they allow for simultaneous maximization of the packing density and minimization of the kissing number. In~\cite{camilla} and~\cite{damir}, the authors showed a particular construction of a family of GWR lattices. Our construction produces a new family of algebraic lattices, which are both, GWR and nearly orthogonal.

\begin{thm} \label{main} Let $n \geq 3$ be prime. There exist infinitely many polynomials $f(x) \in \zed[x]$ of degree $n$ so that $L_f$ is a GWR nearly orthogonal lattice of rank $n$ in~$K_{\real}$, where~$K$ is the splitting field of~$f(x)$; further, $L_f$ contains a basis consisting of minimal vectors. For example, we can take $f(x) = x^n + a_{n-1} x^{n-1} + a_0$, where 
$$|a_0| > \frac{8(n-1)n!}{\sqrt{(n-2)^2 + 16(n-1)} - (n-2)}$$
and $|a_{n-1}| > |a_0| + 1$ are integers.
\end{thm}

\noindent
We prove Theorem~\ref{main} in Section~\ref{pisot}. Our main tools, in addition to Dubickas's result, are Strong Approximation Theorem, Perron's irreducibility criterion for Pisot polynomials and a theorem of Fukshansky and Kogan on coherence of nearly orthogonal well-rounded lattices. More generally, we show that the lattice constructions we seek are often yielded by sets of algebraic conjugates of large Pisot numbers, thus making a connection to the well-studied Pisot polynomials, a special class of which is described in the statement of our Theorem~\ref{main}. We demonstrate some examples of cubic Pisot and non-Pisot polynomials giving rise to WR lattices in Section~\ref{3dim}. Finally, in Section~\ref{det_sec} we compute the determinant of the lattice $L_f$ in situations when the Galois group of $f(x)$ is either cyclic of order $n=\deg(f)$, the full symmetric group $S_n$, or the alternating group $A_n$. 

Finally, we mention that the results of~\cite{robson-1} and~\cite{robson-2} also provide constructions of algebraic lattices of the type $L_f$, but from some very specific polynomials and in the situation when $K$ is a cyclic number field of odd prime degree. In this sense, our results are more general. We are now ready to proceed.
\bigskip

\section{Automorphisms and minimal vectors}
\label{auto_min}

First, we prove Proposition~\ref{aut}, which is very similar to Lemma~6.1 of~\cite{lf_dk}.

\proof[Proof of Proposition~\ref{aut}]
The embeddings of $K$ are precisely the elements of $G$, and for every $\tau \in G$ and $\alpha,\beta \in \M_f$, $\tau(\alpha), \tau(\beta) \in \M_f$ and 
\begin{eqnarray*}
\left< \tau(\alpha), \tau(\beta) \right> & = & \Tr_K(\tau(\alpha \bar{\beta})) =  \sum_{\sigma \in G} \sigma \tau(\alpha \bar{\beta}) \\
& = & \sum_{\sigma \in G} \sigma(\alpha \bar{\beta}) = \Tr_K(\alpha \bar{\beta}) = \left< \alpha, \beta \right>,
\end{eqnarray*}
since right-multiplication by $\tau$ simply permutes elements of $G$. Therefore $G$ is a subgroup of $\Aut(L_f)$.
\endproof

\noindent
Next, we assume that $n$ is prime and prove Theorem~\ref{auto}. Our main tool is the following result of Dubickas, which we state specifically over~$\que$. For an algebraic number $\beta$, we write $\deg(\beta)$ for the degree of the minimal polynomial of $\beta$.

\begin{thm} [\cite{dub}, Theorem~1] \label{a_dub} Let $n$ be prime and $c_1,\dots,c_n \in \que$. Then
$$\beta = c_1 \alpha_1 + \dots + c_n \alpha_n \in \que$$
if and only if $c_1 = \dots = c_n$. If this is not the case, then $\deg(\beta)$ is divisible by $n$.
\end{thm}

We now prove our Theorem~\ref{auto}.

\proof[Proof of Theorem~\ref{auto}]
First, suppose that $a_{n-1} \neq 0$. Notice that in the case $c_1 = \dots = c_n$,
$$c_1 \alpha_1 + \dots + c_n \alpha_n = -c_1 a_{n-1} \neq 0,$$
unless $c_1 = 0$. By Theorem~\ref{a_dub}, if $c_k$'s are not all equal, then this linear combination cannot be $0$, since $0 \in \que$. This implies that $\alpha_1,\dots,\alpha_n$ are $\que$-linearly independent and $\M_f$ is a free $\zed$-module of rank $n$; hence, $L_f$ is a Euclidean lattice of rank $n$. 

On the other hand, suppose that $a_{n-1} = 0$. Again, by Theorem~\ref{a_dub}, the elements $\alpha_1,\dots,\alpha_n$ can satisfy only one linear relation of the form
$$c_1 \alpha_1 + \dots + c_n \alpha_n = 0,$$
the one with all coefficients equal. Therefore $\alpha_n = - \sum_{k=1}^{n-1} \alpha_k$ and $\alpha_1,\dots,\alpha_{n-1}$ are $\que$-linearly independent. Thus, $\M_f$ is a free $\zed$-module of rank $n-1$; hence, $L_f$ is a Euclidean lattice of rank $n-1$. 

Now, let $\beta \in \M_f$, then
\begin{equation}
\label{beta}
\beta = \sum_{k=1}^n c_k \alpha_k,
\end{equation}
for some $c_1,\dots,c_n \in \zed$ and its conjugates are of the form
\begin{equation}
\label{beta1}
\beta_j = \sigma_j(\beta) = \sum_{k=1}^n c_k \sigma_j(\alpha_k) \in \M_f,
\end{equation}
for all $1 \leq j \leq d$. In general, $\beta$ may have degree less than $d$ (and even less than~$n$), in which case some of these $\beta_j$'s may be equal to each other. In particular, if $c_1 = \dots = c_n$, then $\beta = -c_1 a_{n-1} \in \que$ and and hence $\beta_j = \sigma_j(\beta) = \beta$ for every $1 \leq j \leq n$, since $\sigma_j$'s fix $\que$. Thus
$$\left| \Sigma_K(\beta) \right| = \sqrt{d} |c_1| |a_{n-1}|.$$
Taking $c_1=1$, we see that $L_f$ contains a vector of norm $\sqrt{d} |a_{n-1}|$. 

Suppose that $|L_f| < \sqrt{d} |a_{n-1}|$. Then there must exist a $\beta \in M_f$ as in~\eqref{beta} so that not all $c_k$'s are equal and $\left| \Sigma_K(\beta) \right| = |L_f|$. Then, by Theorem~\ref{a_dub}, this $\beta$ has degree divisible by $n$, and thus at least the number of distinct conjugates is divisible by~$n$. All of these conjugates give rise to minimal vectors in $L_f$ via the Minkowski embedding, so the cardinality of $S(L_f)$ is divisible by~$n$.
\endproof

\proof[Proof of Corollary~\ref{n_WR}]
Suppose $n=d$ is prime and $|L_f| < \sqrt{d} |a_{n-1}|$. Let $\bbb \in S(L_f)$ and let $\beta \in \M_f$ be such that $\bbb = \Sigma_K(\beta)$. Then $\beta$ has degree divisible by $n$, by the same argument as in the proof of Theorem~\ref{auto} above. Since $[K:\que]=n$ and $\beta \in K$, we must have $\deg(\beta) = n$. Let $\beta = \beta_1,\beta_2,\dots,\beta_n$ be its conjugates. Theorem~\ref{a_dub} implies again that if 
$$b_1 \beta_1 + \dots b_n \beta_n = 0,$$
then $b_1 = \dots = b_n$. Assume that $\beta_1,\dots, \beta_n$ are linearly dependent, then we must have
$$\beta_1 + \dots + \beta_n = 0.$$
Writing $\beta$ as in~\eqref{beta} and using~\eqref{beta1}, we obtain a relation
\begin{eqnarray*}
0 & = & \sum_{k=1}^n c_k \sigma_1(\alpha_k) + \dots + \sum_{k=1}^n c_k \sigma_n(\alpha_k) =  \sum_{k=1}^n c_k \left( \sigma_1(\alpha_k) + \dots + \sigma_n(\alpha_k) \right) \\
& = & \sum_{k=1}^n c_k \Tr_K(\alpha_k) = -a_{n-1} \sum_{k=1}^n c_k.
\end{eqnarray*}
Since $a_{n-1} \neq 0$, this means that $\sum_{k=1}^n c_k = 0$. Hence, if we assume that $\sum_{k=1}^n c_k \neq 0$, then $\bbb_1,\dots,\bbb_n$ are linearly independent and $L_f$ is WR.
\endproof

\begin{rem} In the reverse direction, working under the assumptions of Corollary~\ref{n_WR} (i.e., $n = d$ be prime, $|L_f| < \sqrt{d} |a_{n-1}|$), suppose that $L_f$ is WR. If $\bbb := \Sigma_K(\beta)$ is a minimal vector, then $\deg(\beta)=n$. Indeed, since $n$ is prime, $\deg(\beta)=1\text{ or } n$. Assume $\deg(\beta)=1$ so $\beta\in \que$, then by Theorem~\ref{a_dub} we must have $c_1= \dots =c_n$, and so $\beta$ is an integer multiple of $\beta'=\sum_{i=1}^n \alpha_i$. However, we assumed that $|L_f| < \sqrt d|a_{n-1}|=\|\bbb'\| \leq \|\bbb\|,$ a contradiction. Thus every minimal vector is of degree $n$.
\end{rem}

\proof[Proof of Corollary~\ref{cor:n_WR-cubic}]
First notice that the assumption $n=d=3$ implies that all three algebraic conjugates $\alpha_1,\alpha_2,\alpha_3$ are real. By Corollary~\ref{n_WR}, it suffices to show that if $\beta=\sum_{k=1}^3 c_k\alpha_k \in \mathcal M_f$ such that $\sum_{k=1}^3 c_k=0$ then $\bbb \not\in S(L_f)$. For any such $\beta$, the vector $\bbb$ is of the form
\begin{align*}
    \bbb = \begin{pmatrix}
        c_1\alpha_1+c_2\alpha_2+c_3\alpha_3\\
        c_1\alpha_3+c_2\alpha_1+c_3\alpha_2\\
        c_1\alpha_2+c_2\alpha_3+c_3\alpha_1,
    \end{pmatrix},
\end{align*}
and $\|\bbb\|^2$ is equal to
\begin{align}
\label{eq:cubic-norm-squared}
(c_1^2+c_2^2+c_3^2)(\alpha_1^2+\alpha_2^2+\alpha_3^2)+2(c_1c_2+c_2c_3+c_1c_3)(\alpha_1\alpha_2+\alpha_1\alpha_3+\alpha_2\alpha_3).
\end{align}
Let $A = \sum_{i=1}^3\alpha_i^2$ and $B = \sum_{i < j}\alpha_i\alpha_j$. Taking $c_1=1$ and $c_2=c_3=0$, we obtain a vector $\baa_1 \in L_f$ with $\|\baa_1\|^2 = A$. On the other hand, if we assume that $\sum_{i=1}^3c_i=0$, \eqref{eq:cubic-norm-squared} simplifies to
\begin{align}
\label{eq:simplified-norm-cubic}
 \|\bbb\|^2 = 2(c_1^2+c_2^2+c_1c_2)(A-B).
\end{align} 
Suppose now that $B = \sum_{i < j}\alpha_i\alpha_j < \frac{1}{2}\sum_{i=1}^3 \alpha_i^2 = A/2$, then $2(A-B) > A$, and since $c_1^2 + c_2^2 + c_1c_2 \geq 1$, we have
$$ \|\bbb\|^2 > A = \|\baa_1\|^2.$$
Thus $\bbb \not\in S(L_f)$, and so $L_f$ is WR by Corollary~\ref{n_WR}.

For the second statement of the corollary, let $A,B$ be as above and assume that $|B| < A/3$. We first show that any $\beta=\sum_{i =1}^3 c_i\alpha_i$ with at most one $c_i=0$  satisfies $\|\bbb\| > \|\baa_1\|$. We can assume without loss of generality that $c_3=0$, so equation~\eqref{eq:cubic-norm-squared} simplifies to 
\begin{align}
\label{eq:c_3=0-simplification}
\|\bbb\|^2 = (c_1^2+c_2^2) \sum_{i=1}^3 \alpha_i^2 + 2(c_1c_2) \sum_{i < j} \alpha_i\alpha_j.
\end{align}
Since $c_1,c_2\in \zed$ and are both nonzero, let $r=c_2/c_1\in \que$ where we suppose $|r| \geq 1$ without loss of generality. Using this substitution, rewrite \eqref{eq:c_3=0-simplification}, 
\begin{align*}
\|\bbb\|^2 = c_1^2 \left( (1+r^2)A+2rB \right) &\geq c_1^2 \left( (1+r^2)A-2|r||B| \right)\\
&\geq 2|r| c_1^2 (A-|B|)\\
&\geq 2(A-|B|)\\
&>4A/3>\|\baa_1\|^2,
\end{align*}
since $1+r^2 \geq 2|r|$ and $2|r|c_1^2 = 2|c_1c_2|\geq 2$, as desired.

Next,  suppose $\beta=\sum_{i=1}^3 c_i\alpha_i$ with $c_i\neq 0$ for all $i$. We can assume without loss of generality that $|c_1|\leq |c_2|\leq |c_3|$, so there exists $r,s \in \que$ with $1\leq |r|,1\leq |s|$, such that $c_2 = r c_1$ and $c_3 = s c_1$. With this substitution, \eqref{eq:cubic-norm-squared} becomes 
\begin{align}
\label{eq:nonzero-simplification}
 \|\bbb\|^2 = c_1^2(1+r^2+s^2)A+2c_1^2(r+s+rs)B
\end{align}
If either $r<0$ or $s<0$, then $1+r^2+s^2 \geq 2|r+s+rs|$, and so the same argument as above holds. If, on the other hand, both $0<r$ and $0<s$, then we still have  $1+r^2+s^2 \geq |r+s+rs|$, and \eqref{eq:nonzero-simplification} gives
\begin{align*}
\|\bbb\|^2 \geq c_1^2(|r+s+rs|)(A-2|B|) \geq 3 (A-2|B|) > A = \|\baa_1\|^2.
\end{align*}
Putting these observations together, we see that $\baa_1,\baa_2,\baa_3$ must be minimal vectors. This completes the proof.
\endproof

Inequalities similar to Corollary~\ref{cor:n_WR-cubic} guaranteeing that $\baa_i$'s are minimal vectors in $L_f$ can potentially be worked out in higher dimensions too, although the arguments become increasingly complicated with the growing dimension.  
These inequalities are essentially a special case (with the lattices of the special form $L_f$) of the more general Tammela inequalities~\cite{tammela} (see also Section~2/2 of~\cite{achill}), which have been worked out in dimensions $\leq 7$.

\begin{rem} \label{n_not_prime} Our main tool in this section is Theorem~\ref{a_dub}, which applies only in the case when $n$ is prime. What if $n$ is not prime: are there some conditions to ensure that the conjugates $\alpha_1,\dots,\alpha_n$ are linearly independent? In general, this is difficult, however there are a couple other observations we can mention. First, if $a_{n-1} \neq 0$ and $\sum_{i=1}^n c_i \alpha_i \neq 0$ for any $c_1,\dots,c_n$ satisfying $\sum_{i=1}^n c_i \neq 0$, then $\alpha_1,\dots,\alpha_n$ are linearly independent, and so $L_f$ has rank $n$: this is guaranteed by Lemma~2.4 of~\cite{dub1}. Additionally, suppose $n=2p$ for a prime number $p$ and $\alpha_1$ is a {\it Salem number} of degree $n$ (a real algebraic number $> 1$ with all of its conjugates having absolute value no greater than $1$ so that at least one of them has absolute value exactly $1$) so that $a_{n-1} \neq 0$, then $\alpha_1,\dots,\alpha_n$ are again linearly independent, so $L_f$ has rank $n$: this is guaranteed by Corollary~1.2 of~\cite{dub2}.

\end{rem}

\bigskip

\section{Two-dimensional situation}
\label{2dim}

We start with a useful lemma about minimal norm of a 2-dimensional lattice.

\begin{lem} \label{2dim_lem} Let $L$ be a Euclidean lattice of rank $2$ with a basis $\bx,\bwy$. Then
$$|L| = \min \left\{ \|a\bx + b\bwy\| : a,b \in \{0,1,-1\} \right\}.$$
\end{lem} 

\proof
Let $\theta$ be the angle between $\bx$ and $\bwy$. By Lemma~2.1 of~\cite{lf_dk-1}, if $\theta \in [\pi/3,2\pi/3]$, then $|L| = \min \{ \|\bx\|,\|\bwy\| \}$. Assume that this is not the case and let $\min \{ \|\bx\|,\|\bwy\| \} = \|\bx\|$. Then it is easy to check that the angle between $\bx$ and one of $\pm \bx \pm \bwy$ satisfies this condition. Since
$$L = \spn_{\zed} \{ \bx, \bwy \} = \spn_{\zed} \{ \bx, \pm \bx \pm \bwy\}$$
for any choice of $\pm$ signs, the result follows, again by Lemma~2.1 of~\cite{lf_dk-1}.
\endproof

Let $f(x) = x^2+a_1x+a_0 \in \zed[x]$ be an irreducible quadratic polynomial with discriminant $D=a_1^2-4a_0$, then its roots are
$$\alpha_1 = \frac{-a_1+\sqrt{D}}{2},\ \alpha_2 = \frac{-a_1-\sqrt{D}}{2},$$
and the Galois group consists of the identity map $\sigma_1$ and the map $\sigma_2 : \sqrt{D} \mapsto -\sqrt{D}$. Consider the lattice $L_f = C_f \zed^2$, where
$$C_f = \begin{pmatrix} \frac{-a_1+\sqrt{D}}{2} & \frac{-a_1-\sqrt{D}}{2} \\ \frac{-a_1-\sqrt{D}}{2} & \frac{-a_1+\sqrt{D}}{2} \end{pmatrix},$$
i.e., $L_f = \spn_{\zed} \left\{ \baa_1,\baa_2 \right\}$. If $a_1 = 0$, then $L_f$ has rank 1, hence assume $a_1 \neq 0$. Define 
$$\beta = \alpha_1-\alpha_2 = \sqrt{D},\ \gamma = \alpha_1 + \alpha_2 = -a_1 = \Tr_K(\alpha_1),$$
then $\bbb = \baa_1-\baa_2 = \begin{pmatrix} \sqrt{D} \\ -\sqrt{D} \end{pmatrix} \in L_f,\ \bgg = \baa_1+\baa_2 = \begin{pmatrix} -a_1 \\ -a_1 \end{pmatrix} \in L_f$. Notice that
$$\|\bbb\| =\sqrt{2|D|},\ \|\bgg\| = \sqrt{2} |a_1|,$$
whereas $\|\baa_1\| = \|\baa_2\| = \frac{\sqrt{a_1^2+|D|}}{\sqrt{2}}$. Then, by Lemma~\ref{2dim_lem},
$$|L_f| = \min \{ \|\baa_1\|, \|\bbb\|, \|\bgg\| \},$$
since $\|\baa_1\| = \|\baa_2\|.$

First, suppose that $|L_f| = \|\bgg\|$. If $S(L_f) = \{ \pm \bgg \}$, then $L_f$ is not WR. Assume then that $S(L_f)\neq \{ \pm \bgg \}$. Then it must be the case that either $\|\bbb\|=\|\bgg\|$ or $\|\baa_1\|=\|\bgg\|$. If 
$$\|\bbb\|^2 = 2|D| = 2a_1^2 = \|\bgg\|^2,$$
then there are two possibilities. First, $a_0 = 0$, which would mean that $f(x)$ is reducible. Second, $a_1^2=2a_0$, meaning that $f(x) = x^2+2kx+2k^2$ for some nonzero integer $k$, but in this case $\alpha_{1,2} = -k (1 \pm \sqrt{-1})$ and so
$$\|\baa_1\| = \|\baa_2\| = 2k < 2\sqrt{2} k = \|\bbb\| = \|\bgg\|.$$
Hence, we cannot have both $\bbb,\bgg \in S(L_f)$. If, on the other hand,
$$\|\baa_1\|^2 = \frac{a_1^2+|D|}{2} = 2a_1^2 = \|\bgg\|^2,$$
then $|D| = 3a_1^2$. This is possible if either $a_1^2 = -2a_0$, or if $a_1^2=a_0$. In the first case, $f(x) = x^2 + 2kx - 2k^2$ for nonzero $k \in \zed$ and
$$L_f = k \begin{pmatrix} -1+\sqrt{3} & -1-\sqrt{3} \\ -1-\sqrt{3} & -1+\sqrt{3} \end{pmatrix} \zed^2.$$
In the second case, $f(x) = x^2 + kx + k^2$ for nonzero $k \in \zed$ and
$$L_f = \frac{k}{2} \begin{pmatrix} -1+ \sqrt{-3} & -1-\sqrt{-3} \\ -1-\sqrt{-3} & -1+\sqrt{-3} \end{pmatrix} \zed^2.$$
In both of these cases, $S(L_f) = \{ \pm \baa_1, \pm \baa_2, \pm \bgg \}$ and the lattice $L_f$ is {\it similar} to an isometric copy of the hexagonal lattice 
$$\Lambda_h := \begin{pmatrix} 1 & 1/2 \\ 0 & \sqrt{3}/2 \end{pmatrix} \zed^2$$
in $K_{\real}$. Recall that two lattices $L, M$ are similar $L = cU M$ for a positive constant $c$ and a real orthogonal matrix $U$.

Next, assume then that $|L_f| < \|\bgg\|$, so the conditions of Corollary~\ref{n_WR} are satisfied. Suppose that $|L_f| = \|\bbb\|$. If $S(L_f) = \{ \pm \bbb \}$, then $L_f$ is not WR. Assume then $S(L_f) \neq \{ \pm \bbb \}$, then we must have $\|\bbb\| = \|\baa_1\| = \|\baa_2\|$. This is possible either if $f(x) = x^2+6kx+6k^2$ for nonzero $k \in \zed$, in which case
$$L_f = \sqrt{3} k \begin{pmatrix} 1-\sqrt{3} & -1-\sqrt{3} \\ -1-\sqrt{3} & 1-\sqrt{3} \end{pmatrix} \zed^2,$$
or $f(x) = x^2+3kx+3k^2$ for nonzero $k \in \zed$, in which case
$$L_f = \frac{\sqrt{-3} k}{2} \begin{pmatrix} 1+ \sqrt{-3} & -1+\sqrt{-3} \\ -1+\sqrt{-3} & 1+\sqrt{-3} \end{pmatrix} \zed^2.$$
Again, in both of these cases, $S(L_f) = \{ \pm \baa_1, \pm \baa_2, \pm \bbb \}$ and the lattice $L_f$ is similar to the hexagonal lattice $\Lambda_h$.

Finally, if $|L_f| = \|\baa_1\| < \|\bbb\|, \|\bgg\|$, then $S(L_f) = \{ \pm \baa_1, \pm \baa_2 \}$ and $L_f$ is WR, as guaranteed by Corollary~\ref{n_WR}. All of the above situations are possible. Let us consider some examples. 

First, take $f_1(x) = x^2+x-2$, then $D=7$ and
$$L_{f_1} = \begin{pmatrix} \frac{-1+\sqrt{7}}{2} & \frac{-1-\sqrt{7}}{2} \\ \frac{-1-\sqrt{7}}{2} & \frac{-1+\sqrt{7}}{2} \end{pmatrix} \zed^2$$
with $|L_{f_1}| = \|\bgg\| = \sqrt{2}$. 

Next, consider $f_2(x) = x^2+5x+5$, then $D=5$ and 
$$L_{f_2} = \begin{pmatrix} \frac{-5+\sqrt{5}}{2} & \frac{-5-\sqrt{5}}{2} \\ \frac{-5-\sqrt{5}}{2} & \frac{-5+\sqrt{5}}{2} \end{pmatrix}  \zed^2$$
with $|L_{f_2}| = \|\bbb\| = \sqrt{10}$. 

Finally, let $f_3(x) = x^2-2x-1$, then $D=8$ and
$$L_{f_3} = \begin{pmatrix} 1+\sqrt{2} & 1-\sqrt{2} \\ 1-\sqrt{2} & 1+\sqrt{2} \end{pmatrix}  \zed^2$$
with $|L_{f_3}| = \|\baa_1\| = \sqrt{6}$. Thus, $L_{f_3}$ is WR.
\smallskip

More generally, we have the following criterion for planar lattices.

\begin{thm} \label{Lf2} Let $f(x) = x^2+a_1x+a_0 \in \zed[x]$ be an irreducible quadratic polynomial with discriminant $D = a_1^2-4a_0$ and the corresponding algebraic lattice
$$L_f = \begin{pmatrix} \frac{-a_1+\sqrt{D}}{2} & \frac{-a_1-\sqrt{D}}{2} \\ \frac{-a_1-\sqrt{D}}{2} & \frac{-a_1+\sqrt{D}}{2} \end{pmatrix} \zed^2.$$
If for some nonzero integer $k$,
$$f(x) \in \left\{ x^2+2kx-2k^2,\ x^2+kx+k^2,\ x^2+6kx+6k^2,\ x^2+3k+3k^2 \right\},$$
then $L_f$ is similar to the hexagonal lattice. Otherwise, $L_f$ is well-rounded if and only if $a_1^2 > 2 \max \{ 3a_0, -a_0 \}$. If this is the case,
$$S(L_f) = \left\{ \pm \begin{pmatrix} \frac{-a_1+\sqrt{D}}{2} \\ \frac{-a_1-\sqrt{D}}{2} \end{pmatrix}, \pm \begin{pmatrix} \frac{-a_1-\sqrt{D}}{2} \\ \frac{-a_1+\sqrt{D}}{2} \end{pmatrix} \right\}$$
and $\Aut(L_f) \cong \zed/2\zed \times \zed/2\zed$ with one of the isomorphic copies of $\zed/2\zed$ coming from the Galois action and the other from the $\pm$ signs.
\end{thm}

\proof
The situations when $L_f$ is similar to the hexagonal lattice are discussed above, so assume this is not the case. Then $\bbb,\bgg \notin S(L_f)$. Let $\alpha_1 = \frac{-a_1+\sqrt{D}}{2}$, $\alpha_2 = \frac{-a_1-\sqrt{D}}{2}$. By Lemma~2.1 of~\cite{lf_dk-1}, $\baa_1, \baa_2$ are minimal vectors if and only if the angle the angle $\theta$ between them is in the interval $[\pi/3,2\pi/3]$, and $\theta \neq \pi/3,2\pi/3$ since otherwise we would have $\bbb$ or $\bgg$ in $S(L_f)$. A simple calculation shows that
$$\cos \theta = \frac{a_1^2-|D|}{a_1^2+|D|}.$$
Hence, we want
$$-1/2 < \frac{a_1^2-|D|}{a_1^2+|D|} < 1/2,$$
meaning that $|D|/3 < a_1^2 < 3|D|$. Solving these inequalities, we obtain the condition
$$a_1^2 > \max \{ 6a_0, -2a_0 \}.$$
Assume this is the case, so the lattice $L_f$ is WR. Let $x_1,x_2 \in \zed$ so $x_1\baa_1+x_2\baa_2 \in L_f$. In addition to the identity map, the automorphisms of $L_f$ are given by $\tau_1 : x_1\baa_1+x_2\baa_2 \mapsto -x_1\baa_1-x_2\baa_2$, the sign change, $\tau_2 : x_1\baa_1+x_2\baa_2 \mapsto x_1\baa_2+x_2\baa_1$, the Galois action, and their composition $\tau_2 \tau_2$. These are all elements of order $2$ commuting with each other. Hence $\Aut(L_f) \cong \zed/2\zed \times \zed/2\zed$.
\endproof

\bigskip

\section{Pisot lattices}
\label{pisot}

In this section we demonstrate explicit constructions of well-rounded lattices of the form $L_f$ and prove Theorem~\ref{main}. We refer the reader to~\cite{nark} for a detailed exposition of the algebraic number theory machinery that we use in this section. Let $K$ be a number field of degree $d \geq 2$ and $M(K) = M_{\infty}(K) \cup M_0(K)$ set of places of $K$, where $M_{\infty}(K)$ stands for the set of archimedean places and $M_0(K)$ for the set of non-archimedean places. Let $\sigma_1,\dots,\sigma_{r_1}$ be the real embeddings of $K$ and $\sigma_{r_1+1},\dots,\sigma_{r_1+2r_2}$ be the complex embeddings coming in conjugate pairs, so that $\sigma_{r_1+r_2+j} = \bar{\sigma}_{r_2+j}$ for each $1 \leq j \leq r_2$. Then $d=r_1+2r_2$. We can write $M_{\infty}(K) = \{ v_1,\dots,v_{r_1}, u_1,\dots,u_{r_2} \}$ and choose representative absolute values as
$$|\alpha|_{v_i} = |\sigma_i(\alpha)|,\ |\alpha|_{u_j} = |\sigma_{r_1+j}(\alpha)|,$$
for each $\alpha \in K$, $1 \leq i \leq r_1$, $1 \leq j \leq r_2$, where $|\ |$ stands for the usual absolute value on $\real$ or $\cee$, respectively. We choose the non-archimedean absolute values to extend the corresponding $p$-adic absolute values on~$\que$. For each place $w \in M(K)$, let $d_w = [K_w:\que_w]$ be the local degree at $w$, then for each place $u \in M(\que)$, 
\begin{equation}
\label{local}
\sum_{v \in M(K), v \mid u} d_v = d,
\end{equation}
and the product formula reads as
\begin{equation}
\label{prod}
\prod_{w \in M(K)} |\alpha|_w^{d_w} = 1.
\end{equation}
From here on, assume that $K$ has at least two archimedean places. Fix an archimedean place $w \in M_{\infty}(K)$ corresponding to some embedding $\sigma_k$ and let $\eps > 0$. Then, by Strong Approximation Theorem, there exists~$\alpha \in K$ such that
\begin{equation}
\label{strong}
|\alpha|_v \leq \eps\ \forall\ v \in M_{\infty}(K) \setminus \{ w \},\ |\alpha|_v \leq 1\ \forall\ v \in M_0(K).
\end{equation}
Notice that the condition $|\alpha|_v \leq 1\ \forall\ v \in M_0(K)$ is equivalent to $\alpha \in \O_K$. Combining~\eqref{local}, \eqref{prod} and~\eqref{strong}, we obtain
\begin{equation}
\label{strong-1}
|\alpha|^{d_w}_w \geq \frac{1}{\eps^{d-d_w}}.
\end{equation}
We want to remark that what we are using here is a limited version of Strong Approximation Theorem, which essentially follows from Minkowski's Linear Forms Theorem. For a real number $T \geq 1$, define
\begin{equation}
\label{T_set}
S_{K,w}(T) = \left\{ \alpha \in \O_K : |\alpha|_v \leq 1\ \forall\ v \in M_{\infty} (K) \setminus \{w\},\ |\alpha|_w \geq T \right\}.
\end{equation}
Then~\eqref{strong-1} implies that $S_{K,w}(T)$ is an infinite set. Indeed, suppose there were only finitely many such $\alpha$, then let
$$T_0 = \max \left\{ |\alpha|_w : |\alpha|_w\geq T \right\}.$$
Take $\eps < T_0^{-\frac{d_w}{d-d_w}}$ and take $\alpha$ satisfying~\eqref{strong-1}, then $|\alpha|^{d_w}_w > T_0$, a contradiction. 

From here on, let the number field $K$ as above be Galois over $\que$ of degree $d \geq 3$ and let $3 \leq n \leq d$. Define
\begin{equation}
\label{Tdn}
T_{d,n} = \frac{8d(n-1)}{\sqrt{(n-2)^2 + 16(n-1)} - (n-2)},
\end{equation}
and let $\alpha \in S_{K,w}(T_{d,n})$ be of degree $n$. Let $f_{\alpha}(x) \in \zed[x]$ be the minimal polynomial of $\alpha$ and let $\alpha_1,\dots,\alpha_n \in \O_K$ be its algebraic conjugates ordered so that $\alpha_k = \sigma_k(\alpha)$. Thus
\begin{equation}
\label{alphaT}
|\alpha|_w = |\alpha_1| \geq T_{d,n},\ |\alpha_2|, \dots, |\alpha_n| \leq 1,
\end{equation}
where $|\ |$ is the usual absolute value on $\real$ or $\cee$. Let $1 \leq m \leq n$ be such that $\alpha_1,\dots,\alpha_m$ are linearly independent over $\que$. Then $\M_{\alpha} := \spn_{\zed} \{ \alpha_1,\dots,\alpha_m \} \subset \O_K$ is a free $\zed$-module of rank $m$ and hence $L_{\alpha} := \Sigma_K(\M_{\alpha}) \subset K_{\real}$ is a lattice of rank~$m$.

\begin{prop} \label{strong_lattice} For $\alpha$ satisfying~\eqref{alphaT}, the lattice $L_{\alpha}$ is nearly orthogonal and generic well-rounded. Further, it contains a basis consisting of minimal vectors.
\end{prop}

\proof
For each $1 \leq i \leq m$, the vector $\baa_i$ has $d$ coordinates, precisely $d/n$ of which are equal to $\alpha_1$. Further, if $i \neq j$, then $\baa_i$ and $\baa_j$ cannot have $\alpha_1$ in the same coordinate. For the product of two coordinates equal to, say, $\alpha_k$ and $\alpha_k$, we have $|\alpha_k \alpha_l| \leq 1$, if $k,l \neq 1$; if $k=1$ and $l \neq 1$, then $|\alpha_k \alpha_l| \leq |\alpha_1|$. This means that
$$|\left< \baa_i, \baa_j \right>| \leq \frac{d}{n} |\alpha_1| + \left( d- \frac{d}{n} \right),$$
whereas $\|\baa_i\| = \|\baa_j\| \geq |\alpha_1|$. Putting these observations together, we see that the {\it coherence} of the pair of vectors $\baa_i,\baa_j$, defined as the absolute value of the cosine of the angle between them is given by
\begin{eqnarray}
\label{coh_bnd}
\frac{ |\left< \baa_i, \baa_j \right>| }{\|\baa_i\| \|\baa_j\|} & \leq &  \frac{d}{|\alpha_1|} \leq \frac{d}{T_n} = \frac{\sqrt{(n-2)^2 + 16(n-1)} - (n-2)}{8(n-1)} \\
& \leq & \frac{\sqrt{(m-2)^2 + 16(m-1)} - (m-2)}{8(m-1)}.
\end{eqnarray}
The conclusion now follows by Theorem~1.6 of~\cite{lf_dk-1}: while this theorem is stated for a full-rank lattice in $\real^d$, the argument applies verbatim to our situation of $L_{\alpha}$ being a full-rank lattice in the $m$-dimensional subspace~$\spn_{\real} L_{\alpha}$ of $\real^d$.
\endproof

Let us now give explicit examples of lattices as in Proposition~\ref{strong_lattice}, proving Theorem~\ref{main}. First recall that a {\it Pisot number} is a real algebraic integer $\alpha$ greater than $1$ all of whose algebraic conjugates have modulus strictly smaller than $1$. Let $F$ be a real algebraic extension of $\que$ of degree $n$. Then $F$ contains infinitely many Pisot numbers of degree~$n$ (see Theorem~5.2.2 of~\cite{pisot_num}). This in particular implies that there are infinitely many Pisot numbers $\alpha \in F$ such that $|\alpha| > T$ for any $T$. Indeed, if this was not true, then the set
$$\{ \alpha \in \O_F : \max_{1 \leq i \leq n} |\sigma_i(\alpha)| \leq T \}$$
would have to be infinite, meaning that the set of points of the lattices $\Sigma_F(\O_F)$ inside of the cube $\{ \bx \in F_{\real} : \max_{1 \leq i \leq n} |x_i| \leq T \}$ is infinite. However, the intersection of a discrete set and a compact set has to be finite.

The minimal polynomial of a Pisot number is called a {\it Pisot polynomial}. Let $f(x) = x^n + \sum_{k=0}^{n-1} a_k x^k \in \zed[x]$  be such that
\begin{equation}
\label{perron_irr}
|a_{n-1}| > 1 + |a_{n-2}| + \dots + |a_1| + |a_0|.
\end{equation}
Then Perron's irreducibility criterion~\cite{perron} guarantees that $f(x)$ is a Pisot polynomial.

\proof[Proof of Theorem~\ref{main}] Let the notation be as in the statement of the theorem. Then $d = [K:\que] \leq n!$ and $f(x)$ is a Pisot polynomial, by~\eqref{perron_irr}. Let $\alpha_1,\dots,\alpha_n \in \O_K$ be the roots of $f(x)$ with $\alpha_1$ being the corresponding Pisot number. Then
$$|\alpha_1| \geq \prod_{i=1}^n |\alpha_i| = |a_0| > T_{d,n},$$
where $T_{d,n}$ is as in~\eqref{Tdn}. Then $L_f = L_{\alpha_1}$ and so it is nearly orthogonal and GWR, by Proposition~\ref{strong_lattice}. Since $a_{n-1} \neq 0$, Theorem~\ref{auto} implies that $L_f$ has rank~$n$.
\endproof

\begin{rem} \label{coh_rem}
Notice, in fact, that Theorem~1.6 of~\cite{lf_dk-1} can be applied more generally than just to Pisot polynomials. Indeed, if the maximal coherence of the collection of vectors $\baa_1,\dots,\baa_n$ is bounded from above as in~\eqref{coh_bnd}, then these vectors still form a minimal basis for the corresponding lattice $L_f$. Consider, for instance, the case $n=d=3$: then~\eqref{coh_bnd} gives the bound of~$0.2965...$, which is a bit weaker than the bound of~$1/3$ coming from part (2) of Corollary~\ref{cor:n_WR-cubic}. On the other hand, the bound of~\eqref{coh_bnd} works for any $n$ and tends to $1/n$ as $n \to \infty$.
\end{rem}

\bigskip

\section{Cubic examples}
\label{3dim}

In this section, we provide some explicit examples of WR lattices coming from irreducible cubic polynomials of the form
\begin{equation}
\label{cubic_p}
f(x) = x^3+ax^2+bx+c
\end{equation}
The table below provides examples of Pisot polynomials of the form~\eqref{cubic_p}. 
\begin{center}
    \begin{tabular}{c|c}
        $a,b,c$ & \text{Galois Group}\\
        \hline
        $\pm 42, 0,\mp 24$ & $\zed/3\zed$\\
        $32, 0, -28$ & $S_3$\\
        $49,0,47$ & $S_3$\\
        $-47, 0,28$ & $S_3$\\
    \end{tabular}
\end{center}
In each case, the lattice is WR by Proposition~\ref{strong_lattice}. For example, for $f(x)=x^3+42x^2-24$, $T_{d,n}=\frac{48}{\sqrt{33} +1}<|\alpha_1|=41.986...$ The Galois group of the polynomial is $\zed/3\zed$ if and only if the discriminant of the polynomial is a square.

\begin{rem}
For our irreducible cubic as above to have Galois group $\zed/3\zed$, it is necessary for $ab<0$. Indeed, the above polynomial has discriminant $-4a^3b-27b^2$, and for all roots to be real we require
$$-4a^3b>27b^2\geq 0,$$
thus $-4a^3b\geq 0$, so $ab<0$.
\end{rem} 

Next, we demonstrate some examples of non-Pisot cubic polynomials of the form~\eqref{cubic_p} with Galois group $\zed/3\zed$.
\begin{center}
    \begin{tabular}{c|c}
        $a,b,c$ & \text{Roots}\\
        \hline
        $3,-6,1$&$2.53209,1.3473,-0.879385$\\
    $5,6,1$&$-0.198062,-3.24698,-1.55496$\\
        $-1, -2, 1$&$1.80194 ,-1.24698,0.445042 $\\
        $-3, 0, 3$&$2.53209,1.3473,-0.879385$\\
     \end{tabular}
\end{center}

\noindent
The fact all of these polynomials have Galois group $\zed/3\zed$ follows from the observation that they have square discriminants, thus we can use Corollary~\ref{n_WR}. In other words, if we show $|L_f|<\sqrt 3|a_{n-1}|$ and $\sum_{i < j}\alpha_i\alpha_j<\frac{1}{2}\sum_{i=1}^3\alpha_i^2$, then the lattice is WR by Corollary~\ref{cor:n_WR-cubic}. Let us go through the first case as an illustrative example. We have $\alpha_1^2+\alpha_2^2+\alpha_3^2=21<3\Tr(\alpha)^2=27$, thus $|L_f|<\sqrt 3 |a_{n-1}|$. Next, $\sum_{i < j}\alpha_i\alpha_j=-6<21/2$, so $L_f$ satisfies Corollary~\ref{cor:n_WR-cubic}.

\bigskip

\section{Determinant}
\label{det_sec}

In this section, we obtain determinant formulas for some specific lattices of the form $L_f$. The first observation is that for a monic irreducible polynomial $f(x) \in \zed[x]$ of degree $n$ with cyclic Galois group $\zed/n\zed$,
$$\det(L_f) = \prod_{j=1}^{n-1} g(\zeta_n^j),$$
where $g(x) := \alpha+\sigma_1(\alpha)x+\sigma_2(\alpha)x^2+\dots+\sigma_{n-1}(\beta)x^{n-1}$ and $\zeta_n = e^{2pi/n}$ is $n$-th primitive root of unity; this is simply the determinant of the corresponding circulant basis matrix (see, e.g.,~\cite{lehmer}). We also provide formulas for the determinant in the cases when the Galois group of $f(x)$ is the full symmetric group $S_n$ or the alternating group $A_n$. We start with two preliminary lemmas.

\begin{lem}\label{lem:entry-of-L_f}
Let $f(x)$ be of degree $n$ with roots $\alpha_1,\dots,\alpha_n$ and Galois group $S_n$. Let $C_f = (\baa_1 \dots \baa_n)$ be the corresponding $n! \times n$ basis matrix for $L_f$, then $ij$-th entry of the Gram matrix $C_f^\top C_f$ is
\begin{align*}
c_{ij}&=\begin{cases}
(n-1)! A & i=j\\
2(n-2)! B & i\neq j,
\end{cases}
\end{align*}
where $A = \sum_{i \leq n}\alpha_i^2$ and $B = \sum_{i < j}\alpha_i\alpha_j$.
\end{lem}

\proof
Let $d=n!$ and $\sigma_1,\dots,\sigma_d$ be the embeddings of $K$, the splitting field of $f(x)$. Then 
$$C_f = \begin{pmatrix}
        \sigma_1(\alpha_1)&\dots&\sigma_1(\alpha_n)\\
        \vdots&\hdots&\vdots\\
        \sigma_d(\alpha_1)&\dots&\sigma_d(\alpha_n) \end{pmatrix},$$
and so the $ij$-th entry of this matrix is
$$c_{ij}=\sum_{k \leq d}\sigma_k(\alpha_i\alpha_j).$$
If $i=j$, we have $c_{ii} = (d/n) \sum_{k \leq n} \alpha_k^2$. Consider $i \neq j$. Since $\alpha_i\alpha_j=\alpha_j\alpha_i$, assume that $i<j$, and thus, there are $\frac{n(n-1)}{2}$ possibilities for $\alpha_i\alpha_j$, $i<j\leq n$. Each of these will appear an equal number of times, and since we are summing over $d=n!$ embeddings, each pair appears $n! / \left( \frac{n(n-1)}{2} \right) = 2(n-2)!$ times. 
\endproof

\begin{lem}\label{lem:entry-of-L_f-alternating} 
Let $f(x)$ be of degree $n$ with roots $\alpha_1,\dots,\alpha_n$ and Galois group $A_n$. With the same notation as Lemma~\ref{lem:entry-of-L_f}, we have
    \begin{align*}
    c_{ij}&=\begin{cases}
    (n-1)! A & i=j\\
    (n-2)! B & i\neq j,
    \end{cases}
    \end{align*}
    where $A = \sum_{i \leq n}\alpha_i^2$ and $B = \sum_{i < j}\alpha_i\alpha_j$.
\end{lem}

\begin{proof}
The argument is the same as in the proof of Lemma~\ref{lem:entry-of-L_f}, keeping in mind that $|A_n|=\frac{1}{2}|S_n|$. Also note that $A_n$ contains all even permutations, and thus contains all even products of disjoint two-cycles.
\end{proof}

\begin{rem} \label{SnAn}
Notice that the argument in the proofs of Lemmas~\ref{lem:entry-of-L_f} and~\ref{lem:entry-of-L_f-alternating} above uses the fact that we sum over all possible combinations of $\alpha_i\alpha_j, i<j\leq n$. This means that the Galois group of $f(x)$ must contain all even products of disjoint $2$-cycles, so that we can simultaneously send $\alpha_i$ to $\alpha_{\ell}$ and $\alpha_j$ to $\alpha_m$ for any $m \neq \ell$. Any subgroup of $S_n$ with this property must contain $A_n$, thus our argument only works when the Galois group of $f(x)$ is $A_n$ or $S_n$.
\end{rem}

\begin{prop}
Let $f(x)$ be of degree $n$ with roots $\alpha_1,\dots,\alpha_n$ and let $A = \sum_{i \leq n}\alpha_i^2$ and $B = \sum_{i < j}\alpha_i\alpha_j$, as above. If the Galois group of $f(x)$ is $S_n$, then 
$$\det(L_f) = \left\{ ((n-2)!)^n (n-1)(A+2B)\left((n-1)A-2B\right)^{n-1} \right\}^{1/2}.$$
If the Galois group of $f(x)$ is $A_n$, then 
$$\det(L_f) = \left\{ ((n-2)!)^n (n-1)(A+B)\left((n-1)A-B\right)^{n-1} \right\}^{1/2},$$
\end{prop}

\proof
Using the notation of Lemmas~\ref{lem:entry-of-L_f} and~\ref{lem:entry-of-L_f-alternating}, observe that in the case of the symmetric group,
$$\det(L_f) = \sqrt{\det(C_f^{\top} C_f)} = \left\{ (n-2)!^n \det \begin{pmatrix} (n-1)A & \dots & 2B \\ \vdots & \ddots & \vdots \\ 2B & \dots & (n-1)A \end{pmatrix} \right\}^{1/2}.$$
Notice that
$$\begin{pmatrix} (n-1)A & \dots & 2B \\ \vdots & \ddots & \vdots \\ 2B & \dots & (n-1)A \end{pmatrix} = ((n-1)A-2B)I_n+2B J_n,$$
where $J_n$ is the all $1$'s matrix. The matrix $2BJ_n$ has eigenvalues $2nB$ with multiplicity $1$ and $0$ with multiplicity $n-1$. Adding $((n-1)A-2B)I_n$ shifts the eigenvalues by the same amount, so eigenvalues are $(n-1)A-2B$ with multiplicity $n-1$ and $(n-1)A-2B+2nB=(n-1)(A+2B)$ with multiplicity~$1$. Then the determinant is obtained by taking the product of these eigenvalues. The argument is the same when the Galois group is $A_n$ with the entries in the matrix above being $B$ instead of~$2B$.
\endproof
\bigskip

\section{Conclusion}
\label{conclusion}

In this paper we presented a construction of lattices of prime rank $n$ in $d$-dimensional Euclidean spaces, $d \leq n!$, generated by roots of monic irreducible polynomials of degree $n$ with integer coefficients. We are especially interested in the conditions under which such algebraic lattices are generic well-rounded and have large automorphism groups. In fact, we believe that our methods should work well to produce such lattices even in the case when $n$ is composite, however it is more difficult to establish linear independence of these roots in the composite case. Indeed, we rely on the result of Dubickas~\cite{dub}, which guarantees linear independence of these roots whenever the trace of their minimal polynomial is nonzero in the case of prime degree. An interesting direction for future research would be to extend such linear independence criteria beyond the prime degree situation and then apply our method to produce a wider class of GWR algebraic lattices with large automorphism groups.
\bigskip

\noindent
{\bf Acknowledgements:} We wish to thank the anonymous referee for a thorough reading and multiple suggestions that improved the quality of this paper.

\bigskip

\bibliographystyle{plain}  

\end{document}